\documentclass{article}
\usepackage[utf8]{inputenc}
\usepackage[]{graphicx}
\usepackage[]{amsmath,amssymb}
\usepackage[]{amsthm,enumerate}
\usepackage{verbatim,bm,comment}
\usepackage{color}
\usepackage{comment}

\newtheorem{theorem}{Theorem}[section]
\newtheorem{lemma}[theorem]{Lemma}
\newtheorem{proposition}[theorem]{Proposition}

\theoremstyle{definition}

\newtheorem{example}[theorem]{Example}

\theoremstyle{remark}
\newtheorem{remark}[theorem]{Remark}

\newtheorem{problem}[theorem]{Problem}

\DeclareMathOperator{\rank}{rank}

\title{
On the Smith normal form of a skew-symmetric D-optimal design of order $n\equiv 2\pmod{4}$}
\author{ Gary Greaves\thanks{G.G. was supported by  the Singapore
Ministry of Education Academic Research Fund (Tier 1);  grant number:  RG127/16.}\\
  School of Physical and Mathematical Sciences, \\
  Nanyang Technological University, \\
   21 Nanyang Link, Singapore 637371\\
  {\tt grwgrvs@gmail.com}
\and
 Sho Suda\thanks{S.S. was supported by JSPS KAKENHI; grant number: 15K21075. } \\
 Department of Mathematics Education, \\  Aichi University of Education, \\ Kariya, 448-8542, Japan \\
 \tt{suda@auecc.aichi-edu.ac.jp}
}
\begin{document}
\maketitle
\begin{abstract}
We show that the Smith normal form of a skew-symmetric D-optimal design of order $n\equiv 2\pmod{4}$
is determined by its order.
Furthermore, we show that the Smith normal form of such a design can be written explicitly in terms of the order $n$, thereby proving a recent conjecture of Armario.
We apply our result to show that certain D-optimal designs of order $n\equiv 2\pmod{4}$ 
are not equivalent to any skew-symmetric D-optimal design. 
We also provide a correction to a result in the literature on the Smith normal form of D-optimal designs.
\end{abstract}

\section{Introduction} 
\label{sec:introduction}

The classification of 
$\{1,-1\}$-matrices of order $n$ having largest possible determinants 
is one of the most important problems in design theory.
Such matrices are known as \textbf{D-optimal designs}.
D-optimal designs of order $n \equiv 0 \pmod 4$ and $n \equiv 2 \pmod 4$ are, respectively, called Hadamard matrices and EW matrices.
For such matrices, the \textit{Smith normal form} is a useful invariant that is used to distinguish between those of the same order.
The Smith normal form of a skew-symmetric Hadamard matrix is well-known~\cite{MW,HK}.
In this paper we determine the Smith normal form of skew-symmetric EW matrices.

Two $n\times n$ integer matrices $M$ and $N$ are called \textbf{equivalent} if there exist $n\times n$ unimodular integer matrices $P$ and $Q$ such that $N=PMQ$. 

Let $M$ be an $n\times n$ integer matrix of rank $r$.  
Then $M$ is equivalent to a diagonal integer matrix  
\[
		\operatorname{diag}[m_1,m_2,\ldots,m_r,m_{r+1},\ldots,m_n], 
	\]
satisfying $m_i \;|\; m_{i+1}$ for all $i \in \{1,\dots, n-1\}$ 
 and $m_{r+1}=0$.
This diagonal matrix is called the \textbf{Smith normal form} of $M$ and the values $m_1,\ldots,m_n$ are called the \textbf{invariant factors} of $A$.

Throughout, $I_n$, $J_n$, and $O_n$ will (respectively) denote the $n \times n$ identity matrix, all-ones matrix, and all-zeros matrix.
We omit the subscript when the order is apparent and the matrices denoted by $J$ and $O$ may not necessarily be square matrices.
We use $\mathbf{0}_t$ and $\mathbf{1}_t$ to denote the length-$t$ all-zeros and all-ones (column) vectors respectively and we omit the subscript when the length is apparent.

Let $X$ be a $\{1,-1\}$-matrix of order $n$.
Abusing language, we call $X$ \textbf{skew-symmetric} if $X+X^\top = 2I$, i.e., the matrix $X-I$ is skew-symmetric in the usual sense:  $X-I + (X-I)^\top = O$.


A famous inequality due to Hadamard~\cite{Hadamard1893} is the following:
 \begin{equation*}
 	\label{ineq:Hadamard}
 	|\det(X)| \leqslant n^{n/2}.
 \end{equation*}
 If equality holds then $X$ is called a \textbf{Hadamard matrix}.
Hadamard's inequality can be improved if we restrict to matrices whose orders are not divisible by $4$.
Indeed, Ehlich~\cite{Ehlich1964} and Wojtas~\cite{Wojtas1964} independently showed that for a $\{1,-1\}$-matrix $X$ of order $n \equiv 2\pmod{4}$, Hadamard's inequality can be strengthened to
\begin{align}\label{ineq:EW}
|\det(X)| \leqslant 2(n-1)(n-2)^{(n-2)/2}.
\end{align}
Moreover, there exists a $\{1,-1\}$-matrix achieving equality in \eqref{ineq:EW} if and only if there exists a $\{1,-1\}$-matrix $B$ such that
\begin{align}\label{eq:ew}
BB^\top=B^\top B=
\begin{pmatrix}
(n-2)I+2J & O_{n/2} \\
O_{n/2} & (n-2)I+2J
\end{pmatrix}.
\end{align}
A $\{1,-1\}$-matrix  $B$ of order $n$ is called an \textbf{EW matrix} if it satisfies \eqref{eq:ew}.
Hence, in particular, an EW matrix of order $n$ has determinant $\pm 2(n-1)(n-2)^{(n-2)/2}$.
Working up to equivalence, we may assume that an EW matrix has a positive determinant.
Accordingly, throughout this paper, we assume that an EW matrix $S$ of order $4t+2$ has determinant $\det(S) = (4t)^{2t}(8t+2)$.
%
We refer to the \emph{Handbook of Combinatorial Designs}~\cite{CRCHAndbook2007} for background on Hadamard and EW matrices. 

The Smith normal form of skew-symmetric Hadamard matrices of order $4t$ is uniquely determined \cite{MW,HK}.
\begin{theorem}
A skew-symmetric Hadamard matrix of order $4t$ has the Smith normal form 
 	\[
		\operatorname{diag}[1,\underbrace{2,\dots,2}_{2t-1},\underbrace{2t,\dots,2t}_{2t-1},4t].
	\]
\end{theorem}

Armario~\cite{A} partially determined invariant factors of a skew-symmetric EW matrix of order $4t+2$. 
\begin{theorem}{\rm\cite[Theorem~1.2]{A}}\label{thm:a}
	Let $s_1,\ldots,s_{4t+2}$ be the invariant factors of a skew-symmetric EW matrix of order $4t+2$. 
       Then $s_1=1$ and $s_2=\cdots=s_{2t+1}=2$. 
\end{theorem}

Further, Armario~\cite{A} conjectured the value of the remaining invariant factors of a skew-symmetric EW matrix.
In this note we prove his conjecture.
The following is our main theorem. 
\begin{theorem}\label{thm:main}
	A skew-symmetric EW matrix of order $4t+2$ has Smith normal form 
	\[
		\operatorname{diag}[1,\underbrace{2,\dots,2}_{2t+1},\underbrace{2t,\dots,2t}_{2t-1},2t(4t+1)].
	\]
\end{theorem}

Now we provide an outline of the proof. 
Let $S$ be a skew-symmetric EW matrix of order $n=4t+2$.  
In particular, we have that 
\begin{align}\label{eq:2}
SS^\top=\begin{pmatrix}(n-2)I+2J & O \\ O & (n-2)I+2J \end{pmatrix}.
\end{align} 

\begin{lemma}\label{lem:m1}
Let $S$ be a skew-symmetric EW matrix of order $n=4t+2$ and let $s_1,\ldots,s_{4t+2}$ be the invariant factors of $S$. 
\begin{enumerate}
\item $s_{2t+2}=2$.
\item $s_{4t+2}=2t(4t+1)$.
\end{enumerate}
\end{lemma}

We will give the proof of Lemma~\ref{lem:m1} 
in subsequent sections. 
Using Lemma~\ref{lem:m1} together with Theorem~\ref{thm:a}, we can determine the invariant factors of a skew-symmetric EW matrix.
Indeed, we prove Theorem~\ref{thm:main} as follows.
\begin{proof}[Proof of Theorem~\ref{thm:main}]
Since the other invariant factors have been determined in Theorem~\ref{thm:a} and Lemma~\ref{lem:m1}, it remains to show $s_{2t+3}=\dots=s_{4t+1}=2t$. 
Using \eqref{eq:ew}, we see that $\det(S)=s_1\cdots s_{4t+2}=2^{2t+1}(2t)^{2t}(4t+1)$.
Hence, by Theorem~\ref{thm:a} and Lemma~\ref{lem:m1}, we have 
\begin{align}\label{eq:m}
s_{2t+3}\cdots s_{4t+1}=(2t)^{2t-1}.
\end{align}
\paragraph{Claim 1:} $s_{4t+1}$ divides $2t$. 
Since $s_2=2$ divides $s_{i}$ for each $i\geqslant 2$, we may set $s_i=2s^\prime_i$ for $2\leqslant i\leqslant 4t+2$. 
We have $s^\prime_{4t+1}$ divides $s^\prime_{4t+2}=t(4t+1)$ and $\text{gcd}(t,4t+1)=1$.
Furthermore, since $4s^\prime_{4t+1} s^\prime_{4t+2} = 4s^\prime_{4t+1} t(4t+1)$ divides $2(4t+1)(4t)^{2t}$ we must have that $s^\prime_{4t+1}$ divides $t$.
Therefore $s_{4t+1}$ is a divisor of $2t$.

Since $s_i$ divides $s_{i+1}$ for each $i$ and $s_{4t+1}\leqslant 2t$, by Claim~1, 
we have that $s_{i}\leqslant 2t$ for $i=2t+3,\dots,4t$, from which it follows that 
\begin{align}\label{eq:m1}
s_{2t+3}\cdots s_{4t+1}\leqslant(2t)^{2t-1}.
\end{align}
By \eqref{eq:m}, equality must hold in \eqref{eq:m1}. 
Therefore $s_{2t+3}=\dots=s_{4t+1}=2t$. 
\end{proof}

Our methods for proving Lemma~\ref{lem:m1} involve a tournament matrix that one associates with an EW matrix, called an \emph{EW tournament matrix} (see Section~\ref{sec:pre} for its definition).
Even though it is not needed for the proof of Lemma~\ref{lem:m1}, we also determine the Smith normal form of an EW tournament matrix (see Theorem~\ref{thm:snfa}).

The organisation of this paper is as follows. 
In Section~\ref{sec:pre} we prepare the basic tools for the proof of our main result. 
We will prove Lemma~\ref{lem:m1} (i) and (ii) in Section~\ref{sec:2t2} and Section~\ref{sec:4t2}, respectively. 
In Section~\ref{sec:ex}, we provide examples of EW matrices with Smith normal forms different from Theorem~\ref{thm:main}.   
Note that the EW matrix in Example~\ref{ex:2} is a counterexample to \cite[Corollary 2]{KMS}.
We 
attempt to deal with the counterexample to \cite[Corollary 2]{KMS}.
In Section~\ref{appendix:a} we provide a revised result to \cite[Corollary 2]{KMS}, and we pursue the method demonstrated in Section~\ref{appendix:a} to obtain a result for the case that $4t+1$ is a square of a prime in Section~\ref{appendix:b}.  
Finally, in Section~\ref{sec:snftournament} we determine the Smith normal form of an EW tournament matrix.

\section{Preliminaries}\label{sec:pre}
In this section we collect some facts about skew-symmetric EW matrices and their invariant factors, which we will use in later sections. 

Let $X$ be an EW matrix of order $n$.
It is well-known that $2n-2$ must be a sum of squares~\cite{CRCHAndbook2007}. 
Furthermore, Armario and Frau \cite{AF} showed that $2n-3$ must be a square if $X$ is skew-symmetric. 
Indeed, they showed the following. 
\begin{lemma}{\rm\cite[Proposition~2.4]{AF}}\label{lem:o}
Suppose there exists a skew-symmetric EW matrix of order $4t+2$.
Then $\sqrt{8t+1}$ is an odd integer. 
\end{lemma}

A \textbf{tournament matrix} is a $\{0,1\}$-matrix $A$ such that $A+A^\top=J-I$.  
A tournament matrix $A$ of order $4t+1$ is called an \textbf{EW tournament matrix} if for some $a\in\mathbb{N}\cup\{0\}$,   
\begin{align}
A A^\top&=t(I_{4t+1}+J_{4t+1})+\begin{pmatrix}
-J_t & -J  & -J  &-J\\
-J & J_t  & O  & O \\
-J & O & O_{a}  &-J\\
-J & O & -J  & O_{2t+1-a}
\end{pmatrix},\label{eq:ew1}\\
A^\top A&=t(I_{4t+1}+J_{4t+1})+\begin{pmatrix}
J_t & -J  & O & O \\
-J & -J_t  & -J  & -J \\
O & -J  & J_{a}  & O \\
O & -J  & O & -J_{2t+1-a}
\end{pmatrix}.\label{eq:ew2} 
\end{align}
Since $A+A^\top=J-I$,  we have that
\begin{align}
(A+I) (A^\top+I)&=t(I_{4t+1}+J_{4t+1})+\begin{pmatrix}
O_{t} & O & O & O \\
O & 2J_t  & J  & J \\
O & J  & J_{a}  & O \\
O & J  & O & J_{2t+1-a}
\end{pmatrix}, \label{eq:AAt}\\
(A^\top+I)(A+I) &=t(I_{4t+1}+J_{4t+1})+\begin{pmatrix}
2J_{t} & O & J  & J \\
O & O_{t} & O & O \\
J & O & 2J_{a}  & J \\
J  & O & J & O_{2t+1-a}
\end{pmatrix}.\label{eq:AtA}
\end{align}

A skew-symmetric EW matrix $S$ corresponds to an EW tournament matrix $A$ via the equation \cite{A2015}:
\begin{align}\label{eq:s}
S=\begin{pmatrix}1 & \bf{1}^\top \\ -\bf{1} & I+A-A^\top \end{pmatrix}.
\end{align}
The invariant factors of $S$ and those of $A+I$ are related as follows. 
\begin{lemma}{\rm\cite[Lemma~2,2, Corollary~2.3]{A}}\label{eq:bm}
Let $S$ be a skew-symmetric EW matrix of order $4t+2$ and let $A$ be the EW tournament matrix of order $4t+1$ obtained by \eqref{eq:s}.  
Let $s_1,\ldots,s_{4t+2}$ be the invariant factors of $S$ and $b_1,\ldots,b_{4t+1}$ be the invariant factors of $A+I$. 
Then
\begin{enumerate}
\item  $s_{i+1} = 2b_i$ for $i=1,\ldots,4t+1$. 
\item $\det(A+I)=t^{2t}(4t+1)$. 
\end{enumerate}
\end{lemma}
\begin{proof}
First observe that $S$ is equivalent to $[1]\oplus2(A+I)$.  Indeed, in \eqref{eq:s}, add the first row to all other rows and subtract the first column from all other columns.
Using this equivalence, we see that $s_{i+1} = 2b_i$ for $i=1,\ldots,4t+1$ and $\det(A+I)=(1/2^{4t+1})\det(S)=t^{2t}(4t+1)$. 
\end{proof}

Let $M$ be an integer matrix.
The invariant factors of $M$ can be determined from the the greatest common divisor of all minors of $M$.
Indeed, define $d_i(M)$ to be the greatest common divisor of all $i\times i$ minors of $M$ and set $d_0(M)=1$.
Then the next result is standard. 
 
\begin{lemma}[Corollary 1.20 of \cite{CN}]\label{lem:minor}
Let $M$ be an integer matrix of order $n$ and rank $r$ with invariant factors $m_1,\dots,m_n$. 
Then $m_i=d_i(M)/d_{i-1}(M)$ for $i=1,\dots,r$. 
\end{lemma}

Let $M$ be an invertible $n\times n$ matrix with rows and columns indexed by $\{1,\ldots,n\}$. 
For a subset $I \subseteq \{1,\ldots,n\}$, define $I'=\{1,\ldots,n\}\setminus I$.  
For nonempty subsets $I,J \subseteq \{1,\ldots,n\}$ with $|I|=|J|$, we denote by $M_{I,J}$ the matrix obtained from $M$ by restricting rows to $I$ and columns to $J$ respectively.
 
\begin{lemma}[Page 21 of \cite{HJ}]\label{lem:detminor}
Let $M$ be an invertible $n\times n$ matrix with rows and columns indexed by $\{1,\ldots,n\}$. 
Suppose $I,J \subseteq \{1,\ldots,n\}$ satisfy $|I|=|J| \ne 0$.
Then
\begin{align*}
\det(M_{I,J})=\pm\det(M) \det((M^{-1})_{J',I'}).
\end{align*}
\end{lemma}

Now we record a result about the determinant of an EW tournament matrix.
It is a corollary of \cite[Theorem 1.1]{GS}, 

\begin{lemma}\label{lem:snfa1}
	Let $A$ be an EW tournament matrix of order $4t+1$.
	Then $\det(A)=t^{2t}(4t-1)$. 
\end{lemma}

Since skew-symmetric EW matrices are in correspondence with EW tournament matrices, the following result implies Lemma~\ref{lem:o}.
We will use Theorem~\ref{thm:quadforma} in Section~\ref{sec:snftournament}.

\begin{theorem}
	\label{thm:quadforma}
	Let $A$ be an EW tournament matrix of order $4t+1$.
	Then the integer $a$ in \eqref{eq:ew1} satisfies the quadratic equation $a^2-(2t+1)a+t(t-1)=0$.
\end{theorem}
\begin{proof}
	Begin with \eqref{eq:ew1}, which gives an expression for $AA^\top$.
Using elementary row operations and properties of the determinant, we find that
\begin{align}
\det( AA^\top) &=\begin{pmatrix}
t I_t+(t-1)J_t & (t-1)J  & (t-1)J  &(t-1)J\\
(t-1)J & t I_t+(t+1)J_t  & t J  & t J \\
(t-1)J & t J & t I_a+t J_{a}  & (t-1)J\\
(t-1)J & t J & (t-1)J  & t I_{2t+1-a}+t J_{2t+1-a}
\end{pmatrix}\displaybreak[0]\nonumber\\
&=\det \begin{pmatrix}
t^2 & (t-1){\bf 1}^\top & (t-1){\bf 1}^\top  & (t-1){\bf 1}^\top  & (t-1){\bf 1}^\top\\
{\bf 0}&t I_{t-1} & O  & O  & O \\
-t{\bf 1}& O & t I_t+2J_t  &  J  &  J \\
-t{\bf 1}& O & J & t I_a+J_{a}  & O\\
-t{\bf 1}& O & J & O  & t I_{2t+1-a}+J_{2t+1-a}
\end{pmatrix}\displaybreak[0]\nonumber\\
&=t^{t-1}\det \begin{pmatrix}
t^2 & (t-1){\bf 1}^\top  & (t-1){\bf 1}^\top  & (t-1){\bf 1}^\top\\
-t{\bf 1}&  t I_t+2J_t  &  J  &  J \\
-t{\bf 1}&  J & t I_a+J_{a}  & O\\
-t{\bf 1}&  J & O  & t I_{2t+1-a}+J_{2t+1-a}
\end{pmatrix}\displaybreak[0]\nonumber\\
&=t^{t}\det \begin{pmatrix}
t & (t-1){\bf 1}^\top  & (t-1){\bf 1}^\top  & (t-1){\bf 1}^\top \\
-{\bf 1}&  t I_t+2J_t  &  J  &  J \\
-{\bf 1}&  J & t I_a+J_{a}  & O\\
-{\bf 1}&  J & O  & t I_{2t+1-a}+J_{2t+1-a}
\end{pmatrix}\displaybreak[0]\nonumber\\
&=t^{t}\det \begin{pmatrix}
t & (3t-1){\bf 1}^\top & (2t-1){\bf 1}^\top & (2t-1){\bf 1}^\top\\
-{\bf 1}&  t I_t  &  O  &  O \\
-{\bf 1}&  -J & t I_a  & -J \\
-{\bf 1}&  -J & -J & t I_{2t+1-a}
\end{pmatrix}\displaybreak[0]\nonumber\\
&=t^{t}\det \begin{pmatrix}
4t-1 & (3t-1){\bf 1}^\top & (2t-1){\bf 1}^\top & (2t-1){\bf 1}^\top \\
{\bf 0}&  t I_t  &  O  &  O \\
(-2){\bf 1}&  -J & t I_a  & -J \\
(-2){\bf 1}&  -J & -J & t I_{2t+1-a}
\end{pmatrix}\displaybreak[0]\nonumber\\
&=t^{2t}\det \begin{pmatrix}
4t-1 &  (2t-1){\bf 1}^\top & (2t-1){\bf 1}^\top \\
(-2){\bf 1}&  t I_a  & -J \\
(-2){\bf 1}&  -J & t I_{2t+1-a}
\end{pmatrix}\displaybreak[0]\nonumber\\
&=t^{2t}(4t-1)\det \left( \begin{pmatrix}
 t I_a  & -J \\
  -J & t I_{2t+1-a}
\end{pmatrix}+\frac{2(2t-1)}{4t-1}J\right)\nonumber\\
&=t^{2t}(4t-1)\det\begin{pmatrix}
 t I_a+\frac{2(2t-1)}{4t-1}J  & -\frac{1}{4t-1}J \\
  -\frac{1}{4t-1}J & t I_{2t+1-a}+\frac{2(2t-1)}{4t-1}J
\end{pmatrix}.\label{eq:ew3}
\end{align}
Next, we use the following general equation obtained by the Schur complement \cite[(0.8.5.1)]{HJ}:
\begin{align}\label{eq:det2x2}
\det\begin{pmatrix}
 \alpha I_a+\beta J_a  & \gamma J \\
  \gamma J & \alpha I_b+\beta J_b
\end{pmatrix}=\alpha^{a+b-2}\left( \alpha^2+(a+b)\alpha\beta+ab\beta^2-ab\gamma^2 \right). 
\end{align}
Then, by \eqref{eq:ew3} and \eqref{eq:det2x2}, we have 
\begin{align}\label{eq:eq4}
\det( AA^\top) = t^{4t-1}\left((3-4t)a^2+(8t^2-2t-3)a+t(12t^2-t-2) \right).
\end{align}
By Lemma~\ref{lem:snfa1}, we have $\det(A)=t^{2t}(4t-1)$.
Therefore, using the multiplicativity of the determinant we have
\[
\det( AA^\top) = t^{4t}(4t-1)^2. 
\]
Therefore we can simplify \eqref{eq:eq4} to obtain
\begin{align*}
a^2-(2t+1)a+t(t-1)=0,
\end{align*}
as required. 
\end{proof}

\section{The $(2t+2)$-nd invariant factor of skew-symmetric EW matrices}\label{sec:2t2}
In this section, we provide proof of part (i) of Lemma~\ref{lem:m1}. 
For a prime $p$ and an integer matrix $M$, let $\rank_p(M)$ denote the rank of $M$ over the field of integers of modulo $p$.  
The following was shown in \cite{A}. 
\begin{lemma}{\rm \cite[Lemma~2,4]{A}}
	\label{lem:rankub}
Let $A$ be an EW tournament matrix of order $4t+1$ and let $p$ be a prime that divides $t$.
Then $2t\leqslant \rank_p(A+I)\leqslant 2t+1$. 
\end{lemma}




 Next we provide a strengthening of the above lemma, which is a key lemma to determine the invariant factor $s_{2t+2}$.

\begin{lemma}\label{lem:prank}
	Let $A$ be an EW tournament matrix of order $4t+1$ and let $p$ be a prime that divides $t$.
	Then $\rank_p(A+I) = 2t+1$.
\end{lemma}
\begin{proof}
First we show that $\rank_p(A+I) \geqslant 2t+1$.
Let $d = \dim( \mathrm{rowsp}_p(A+I)\cap \mathrm{rowsp}_p(A^\top+I) )$ be the dimension of the intersection of row spaces over the field of integers of modulo $p$.
Then
\begin{align*}
4t&=\rank_p(J+I)=\rank_p((A+I)+(A^\top+I))\\
&\leqslant \rank_p(A+I)+\rank_p(A^\top+I)-d\\
&=2\rank_p(A+I)-d, 
\end{align*}
that is, we have that $2t+d/2\leqslant \rank_p(A+I)$. 
Hence it suffices to show that $d\geqslant 1$.

Using the assumption that $p$ divides $t$ together with \eqref{eq:AAt} and \eqref{eq:AtA}, it follows that
\begin{align*}
\mathrm{span}\{({\bf 0}^\top_t,{\bf 1}^\top_t,{\bf 1}^\top_{a},{\bf 0}^\top_{2t+1-a}),({\bf 0}^\top_t,{\bf 1}^\top_t,{\bf 0}^\top_{a},{\bf 1}^\top_{2t+1-a})\} &\subseteq \mathrm{rowsp}_p(A+I), \\
\mathrm{span}\{({\bf 1}^\top_t,{\bf 0}^\top_t,{\bf 1}^\top_{a},{\bf 0}^\top_{2t+1-a}),({\bf 1}^\top_t,{\bf 0}^\top_t,{\bf 0}^\top_{a},{\bf 1}^\top_{2t+1-a})\} &\subseteq \mathrm{rowsp}_p(A^\top+I).
\end{align*}
Therefore 
\begin{align*}
\mathrm{span}\{({\bf 0}^\top_t,{\bf 0}^\top_t,{\bf 1}^\top_{a},-{\bf 1}^\top_{2t+1-a})\} \subseteq \mathrm{rowsp}_p(A+I)\cap \mathrm{rowsp}_p(A^\top+I), 
\end{align*}
and thus $d \geqslant 1$.

The inequality $\rank_p(A+I) \leqslant 2t+1$ follows from Lemma~\ref{lem:rankub}.
However, for the sake of completeness, we provide a proof. 
Combine \eqref{eq:AAt} and Sylvester's rank inequality \cite[0.4.5 Rank inequalities (c)]{HJ} to obtain
\begin{align*}
    2\rank_p(A+I)&=\rank_p(A+I)+\rank_p(A^\top+I)\\
    &\leqslant \rank_p((A+I)(A^\top+I))+4t+1\\
    &=4t+3.
\end{align*}
Thus we have $ \rank_p(A+I)\leqslant2t+1$.
\end{proof}

\begin{proposition}\label{prop:b2t1}
Let $A$ be an EW tournament matrix of order $4t+1$ and let $b_1,\ldots,b_{4t+1}$ be the invariant factors of $A+I$.
Then $b_{2t+1}=1$. 
\end{proposition}
\begin{proof}
This result follows from the argument\footnote{In Armario's proof, he derived ``$p^{2t+1}$ divides $4t+1$" from the assumption on $b_{2t}=1$. The result obtained from his argument is, in fact, that ``$p^{2t+2}$ divides $4t+1$".} in \cite[Proof of Theorem~1.2]{A} by replacing the role of $b_{2t}$ with $b_{2t+1}$. 
For the sake of the reader we provide a proof.  
 
Assume that there exists a prime $p$ dividing $b_{2t+1}$. It follows from Lemma~\ref{lem:minor} and Lemma~\ref{eq:bm} that $b_1\cdots b_{4t+1}=\det(A+I)=t^{2t}(4t+1)$. 
Thus $p$ is either 
\begin{enumerate}[(1)]
\item a divisor of $t$, or 
\item not a divisor of $t$ and a divisor of $4t+1$. 
\end{enumerate} 
For the case (1), Lemma~\ref{lem:prank} contradicts the fact that 
\begin{align*}
\rank_p(A+I)=\max\{i\mid \text{$p$ does not divide $b_i$}\}\leqslant 2t.
\end{align*} 
For the case (2), observe that $p$ is odd and hence $p \geqslant 3$.
Furthermore $p$ is a divisor of $b_i$ for $i=2t+1,\ldots,4t+1$. 
Thus $p^{2t+1}$ divides $4t+1$.
But this is impossible since $p^{2t+1}>4t+1$ for all $t\geqslant 1$. 
Therefore we must have $b_{2t+1}=1$. 
\end{proof}
Now we can prove part (i) of Lemma~\ref{lem:m1}.

\begin{proof}[Proof of $s_{2t+2}=2$]
By combining Lemma~\ref{eq:bm} and Proposition~\ref{prop:b2t1} we obtain $s_{2t+2}=2$. 
\end{proof}
\begin{remark}
The equality $b_{2t+1}=1$ implies that $b_1=\cdots=b_{2t}=1$. Thus it follows that $s_2=\cdots=s_{2t+1}=2$, which was already shown in \cite[Theorem~1.2]{A}. 
\end{remark}

\section{The $(4t+2)$-nd invariant factor of skew-symmetric EW matrices}\label{sec:4t2}
In this section, we will show $s_{4t+2}=2t(4t+1)$ for a skew-symmetric EW matrix of order $4t+2$. 

By Lemma~\ref{lem:detminor}, it is enough to calculate the entries of $\det(S)S^{-1}$ explicitly. 
By the equation \eqref{eq:2} it is shown that
\begin{align}\label{eq:1}
S^{-1}=S^\top \left(\frac{1}{4t}I-\frac{1}{4t(4t+1)}\begin{pmatrix}J & O \\ O & J \end{pmatrix}\right).
\end{align}
We use the following lemma.  
\begin{lemma}\label{lem:1}
Set $S=\left ( \begin{smallmatrix}S_{11} & S_{12} \\ S_{21} & S_{22} \end{smallmatrix} \right )$ where each $S_{ij}$ is a $(2t+1) \times (2t+1)$ matrix. 
Then $S_{11}J=S_{22}J=J$ and $S_{12}J=-S_{21}J=\pm\sqrt{8t+1}J$ hold. 
\end{lemma}
\begin{proof}
The eigenvalues of $S$ are $1\pm\sqrt{-1}(4t-1),1\pm\sqrt{-8t-1}$, and the eigenvectors corresponding to $1\pm\sqrt{-8t-1}$ are 
$\begin{pmatrix}\pm\sqrt{-1}\bm{1} \\ \bm{1} \end{pmatrix}$ or $\begin{pmatrix}\mp\sqrt{-1}\bm{1} \\ \bm{1} \end{pmatrix}$, see the proof of \cite[Lemma~4.5]{NS}.
The result now follows from these eigenvalues and eigenvectors. 
\end{proof}
Using Lemma~\ref{lem:1} together with the equation $S+S^\top=2I$, we calculate the right hand side of the equation \eqref{eq:1} as follows:
\begin{align*}
S^{-1}&=(2I-S)\left(\frac{1}{4t}I-\frac{1}{4t(4t+1)}\begin{pmatrix}J & O \\ O & J \end{pmatrix}\right).\\
&=\frac{1}{4t}(2I-S)-\frac{1}{4t(4t+1)}\begin{pmatrix}(2I-S_{11})J & -S_{12}J \\ -S_{21}J & (2I-S_{22})J \end{pmatrix}\\
&=\frac{1}{4t}(2I-S)-\frac{1}{4t(4t+1)}\begin{pmatrix}J & \mp\sqrt{8t+1}J \\ \pm\sqrt{8t+1}J & J \end{pmatrix}\\
&=\frac{1}{4t(4t+1)}\left((4t+1)(2I-S)-\begin{pmatrix}J & \mp\sqrt{8t+1}J \\ \pm\sqrt{8t+1}J & J \end{pmatrix}\right).
\end{align*}
By $\det(S)=(8t+2)(4t)^{2t}$, we obtain
\begin{align*}
\det(S) S^{-1}=2(4t)^{2t-1}\left((4t+1)(2I-S)-\begin{pmatrix}J & \mp\sqrt{8t+1}J \\ \pm\sqrt{8t+1}J & J \end{pmatrix}\right).
\end{align*}

Thus the entries of $\det(S)S^{-1}$ are, up to sign,   
\begin{align*}
2(4t)^{2t-1}\times 4t,\quad2(4t)^{2t-1}\times (4t+2),\quad 2(4t)^{2t-1}\times (4t+1\pm \sqrt{8t+1}). 
\end{align*}
By Lemma~\ref{lem:o}, the expression $\sqrt{8t+1}$ is an odd integer.
Hence $4t,4t+2$, and $4t+1\pm\sqrt{8t+1}$ are all even. 
Moreover, the greatest common divisor among $4t,4t+2,4t+1\pm\sqrt{8t+1}$ is two. 
By Lemma~\ref{lem:minor}, 
we have $d_{4t+1}(S)=4(4t)^{2t-1}$.  
Combining Lemma~\ref{lem:minor} and the equalities $d_{4t+2}(S)=\det(S)=(8t+2)(4t)^{2t}$, yields $s_{4t+2}=d_{4t+2}(S)/d_{4t+1}(S)=2t(4t+1)$. 

\section{Examples}\label{sec:ex}
In this section we provide a couple of examples of EW matrices that have Smith normal forms different from Theorem~\ref{thm:main}.  
\begin{example}\cite[Example 1]{KMS}\label{ex:1}
There exists an EW matrix of order $26$ with the Smith normal form different from Theorem~\ref{thm:main}.  
Let $R$ be the circulant $13 \times 13$ $\{1,-1\}$-matrix with
first row
\begin{align*}
    (1, 1, 1, 1, -1, 1, -1, -1, 1, 1, 1, -1, 1).
\end{align*}
Note that $R$ satisfies that $RR^\top=12I+J$.  
Then $X=\begin{pmatrix}
R & R\\
-R^\top & R^\top
\end{pmatrix}$
is an EW matrix of order $26$, and its Smith normal form is 
\begin{align*}
    \operatorname{diag}[1,\underbrace{2,\dots,2}_{13},\underbrace{12,\dots,12}_{10},60,60].
\end{align*}
Therefore, by Theorem~\ref{thm:main}, the EW matrix $X$ is not equivalent to a skew-symmetric EW matrix of order $26$. 
For a result on the Smith normal form of an EW matrix of order $26$ that have a certain block structure, see Section~\ref{appendix:b}.
\end{example}

\begin{example}\label{ex:2}
The following construction is due to Kharaghani \cite{K}. 
Let $A$ be the $11\times 11$ circulant $\{0,1,-1\}$-matrix with the first row $$(0, -1, 1, -1, -1, -1, 1, 1, 1, -1, 1).$$ 
Note that $A$ is skew-symmetric and $AA^\top=11I-J$. 
Then $X=\begin{pmatrix}
R_1 & R_2\\
-R_2^\top & R_1^\top
\end{pmatrix}$
is an EW matrix of order $66$, where 
\begin{align*}
    R_1&=(A+I_{11})\otimes (J_3-I_3)+(J_{11}-2I_{11})\otimes I_3,\\
    R_2&=(A+I_{11})\otimes (J_3-I_3)+(-A+I_{11})\otimes I_3, 
\end{align*}
and its Smith normal form is
\begin{align*}
    \operatorname{diag}[1,\underbrace{2,\dots,2}_{31},\underbrace{8,\dots,8}_{4},\underbrace{32,\dots,32}_{29},2080].
\end{align*}
Therefore by Theorem~\ref{thm:main}, $X$ is not equivalent to a skew-symmetric EW matrix of order $66$. 
Note that this EW matrix is a counterexample to \cite[Corollary 2]{KMS}.  
We provide a revised result of \cite[Corollary 2]{KMS} in Section~\ref{appendix:a}. 
\end{example}

\section{The Smith normal form of an EW matrix of order $4t+2$ where $4t+1$ is square-free} \label{appendix:a}
Here we provide a revised result of \cite[Corollary 2]{KMS}. 
\begin{lemma}\label{lem:a1}
Let $X$ be an EW matrix of order $4t+2$. 
Let $x_1,\ldots,x_{4t+2}$ be the invariant factors of $X$. 
Then $x_1=1$ and $x_2=2$. 
\end{lemma}
\begin{proof}
The result for $x_1$ is trivial, and that for $x_2$ follows from the fact that some $2\times 2$ minor of $X$ is $\pm2$ and all $2 \times 2$ minors of $X$ are even.  
\end{proof}

In the following theorem, we deal with EW matrices that have a block structure. 
We use the notion $[x]^n$ to mean $\underbrace{x,\ldots,x}_{n \text{ times}}$. 
\begin{theorem}\label{thm:a1}
Let $X$ be an EW matrix $X=\begin{pmatrix} R_1 & R_2 \\ -R_2^\top & R_1^\top \end{pmatrix}
$ of order $4t+2$ such that $R_1J=R_1^\top J=r_1J$, $R_2J=R_2^\top J=r_2J$, and $4t+1$ is square-free. 
Let $x_1,\ldots,x_{4t+2}$ be the invariant factors of $X$. 
Then
\begin{enumerate}
\item $x_{4t+2}=2t(4t+1)$;
\item $x_{4t+1}=2t$;
\item if $t=2^\ell q$ such that $q$ is a square-free odd integer, we have 
\begin{align*}
&(x_2,\ldots,x_{4t})\\
&=([2]^{n_1},\ldots,[2^{k-1}]^{n_{k-1}},[2^k]^{n'_k},[2^k q]^{n''_k},[2^{k+1} q]^{n_{k+1}},\ldots,[2^{\ell+1} q]^{n_{\ell+1}})
\end{align*}
for some positive integer $k$ and non-negative integers 
$n_1,\ldots,n_{k-1},n_k'$, $n_k'',n_{k+1},\ldots,n_\ell+1$ such that $1\leqslant k\leqslant \ell+1$, and
\begin{align*}
n_1+\cdots+n_{k-1}+n_k'&=2t+1,\\
n_k''+n_{k+1}+\cdots+n_{\ell+1}&=2t-2,\\
1\cdot n_1+2\cdot n_2+\cdots+(\ell+1)\cdot n_{\ell+1}&=3+2(\ell+2)(t-1),
\end{align*}
where $n_k=n_k'+n_k''$. 
\end{enumerate} 
\end{theorem}
\begin{proof}[Proof of (i)]
Since $X$ is an EW matrix, $R_1R_1^\top+R_2R_2^\top=4tI+2J$. 
Multiplying all-ones vectors from both sides and dividing by $2t+1$ yield 
\begin{align*}
r_1^2+r_2^2=8t+2.
\end{align*}
Considering this equation modulo $4$, we have $r_1$ and $r_2$ are both odd. 
Since $4t+1$ is square-free, $\text{gcd}(r_1,r_2)=1$. 

Next we calculate $X^{-1}$ as follows:  
\begin{align*}
X^{-1}&=X^\top \left(\frac{1}{4t}I-\frac{1}{4t(4t+1)}\begin{pmatrix}J & O \\ O & J \end{pmatrix}\right)\\
&=\frac{1}{4t}X^\top-\frac{1}{4t(4t+1)}\begin{pmatrix}R_1^\top J & -R_2J \\ R_2^\top J & R_1 J \end{pmatrix}\\
&=\frac{1}{4t}X^\top-\frac{1}{4t(4t+1)}\begin{pmatrix}r_1 J & -r_2 J \\ r_2J & r_1 J \end{pmatrix}\\
&=\frac{1}{4t(4t+1)}\left((4t+1)X^\top-\begin{pmatrix}r_1 J & -r_2 J \\ r_2 J & r_1 J  \end{pmatrix}\right).
\end{align*}
Since $\det(X)=2(4t+1)(4t)^{2t}$, 
\begin{align*}
\det(X)X^{-1}&=2(4t)^{2t-1}\left((4t+1)X^\top-\begin{pmatrix}r_1 J & -r_2 J \\ r_2 J & r_1 J  \end{pmatrix}\right).
\end{align*}
Therefore the entries of $\det(X)X^{-1}$ are $\pm 2(4t)^{2t-1}(4t+1\pm x)$ where $x\in\{r_1,r_2\}$. 
Then the greatest common divisor of $4t+1\pm x$ where $x\in\{r_1,r_2\}$ is 
\begin{align*}
\text{gcd}(4t+1\pm r_1,4t+1\pm r_2)&=\text{gcd}(4t+1+ r_1,2r_1,4t+1+r_2,2r_2)
=2,
\end{align*}
where in the last equation we used that $r_1$ and $r_2$ are both odd integers and $\text{gcd}(r_1,r_2)=1$. 
Therefore $d_{4t+1}(X)=4(4t)^{2t-1}$, and hence we have $x_{4t+2}=\det(X)/d_{4t+1}(X)=2t(4t+1)$.
\end{proof}

Since $x_2=2$ divides $x_{i}$ for each $i\geqslant 2$, we may set $x_i=2x^\prime_i$ for $2\leqslant i\leqslant 4t+2$. 

\begin{proof}[Proof of (ii)]

By Lemma~\ref{lem:detminor}, it is enough to calculate $2\times 2$ minors of $X^{-1}$. 
It is straightforward to check that each possible $2\times 2$ minor of $4t(4t+1)X^{-1}$ has the form $\pm 2(4t+1)\alpha$, where $\alpha$ is an element of the set $S$ given by 
\begin{align*}
	S &= \left \{ \begin{matrix}
		0, &
		   4t, &
		    2 r_1, &
		 2 r_2, &
		     (r_1\pm 1), 
		    \\
					 (r_2\pm 1), &
  		    (4t+1\pm r_1), &
  		 (4t+1\pm r_2), & (r_1\pm r_2), &
		   (4t+2\pm r_1\pm r_2)
	\end{matrix} \right \}.
\end{align*}

By the proof of (i), both $r_1$ and $r_2$ are odd integers.
Hence each element of $S$ is divisible by $2$, and therefore each $2\times 2$ minor of $4t(4t+1)X^{-1}$ is divisible by $4(4t+1)$.

Any principal minor induced on the first $(2t+1)$ rows of $X$ is, up to sign, equal to $2(4t+1)(4t+1\pm r_1)$ or $4(4t+1)r_1$.   
Also any principal minor induced on a subset of the rows  $\{ 2t+2, \dots,4t+2 \}$ is, up to sign, equal to $2(4t+1)(4t+1\pm r_2)$ or $4(4t+1)r_2$. 
Thus the greatest common divisor of the $2\times 2$ minors of $4t(4t+1)S^{-1}$ is $4(4t+1)$. 
By Lemma~\ref{lem:detminor}, $d_{4t}(X)=8(4t)^{2t-2}$.  
Combining Lemma~\ref{lem:minor} and $d_{4t+1}(X)=4(4t)^{2t-1}$, we obtain $x_{4t+1}=d_{4t+1}(X)/d_{4t}(X)=2t$. 
\end{proof}


\begin{proof}[Proof of (iii)]
By (i), (ii), and the equality $\det(X)=x_1\cdots x_{4t+2}$, we have that 
\begin{align}\label{eq:a1}
x^\prime_2 \cdots x^\prime_{4t}=t^{2t-2}.
\end{align} 
Suppose, for a contradiction, that $x^\prime_{2t+2}$ is divisible by an odd integer greater than $1$.  
Let $p$ be any odd prime divisor of $x^\prime_{2t+2}$. 
Then $x^\prime_2 \cdots x^\prime_{4t}=t^{2t-2}$ is divisible by $p^{2t}$.  
However this is impossible; since $t/2^\ell$ is square-free $t^{2t-2}$ cannot have $p^{2t}$ as a divisor. 
Therefore $x^\prime_{2t+2}$ is a power of $2$. 

Next let $p'$ be any prime divisor of $q$. 
By the above, we see that
$x^\prime_{2t+3} \cdots x^\prime_{4t}=2^r q^{2t-2}$ for some $r$.
Therefore, for $i\in\{2t+3,\ldots,4t\}$, $x_{i}$ equals to a power of $2$ times $q$. 
Thus, the invariant factors $x_i$ must have the form given in part (iii).   
\end{proof}

\section{The Smith normal form of an EW matrix of order $4t+2$ where $4t+1=p^2$, for $p$ a prime} \label{appendix:b}

\begin{theorem}\label{thm:3}
Let $X$ be an EW matrix $X=\begin{pmatrix} R_1 & R_2 \\ -R_2^\top & R_1^\top \end{pmatrix}
$ of order $4t+2$ such that $R_1J=R_1^\top J=r_1J$, $R_2J=R_2^\top J=r_2J$, and $4t+1=p^2$ where $p$ is a prime. 
Let $x_1,\ldots,x_{4t+2}$ be the invariant factors of $X$. 
Then 
\begin{enumerate}
\item  $\textrm{gcd}(r_1,r_2)$ is $1$ or $p$;
\item $x_{4t+2}$ is $2t(4t+1)$ or $2tp$;
\item $x_{4t+1}$ is a divisor of $2t$ if $x_{4t+2}=2t(4t+1)$, and $x_{4t+1}$ is a divisor of $2tp$ and $p$ divides $x_{4t+1}$ if $x_{4t+2}=2tp$;
\item if $t$ is square-free, then the Smith normal form of $X$ is either 
\begin{align*}
		&\operatorname{diag}[1,\underbrace{2,\dots,2}_{2t+1},\underbrace{2t,\dots,2t}_{2t-1},2t(4t+1)], \text{ or }\\
&\operatorname{diag}[1,\underbrace{2,\dots,2}_{2t+1},\underbrace{2t,\dots,2t}_{2t-2},2tp,2tp]. 
\end{align*}

\end{enumerate} 
\end{theorem}
\begin{proof}[Proof of (i)]
By the same argument with Theorem~\ref{thm:a1}, we have
\begin{align*}
r_1^2+r_2^2=8t+2=2p^2.
\end{align*}
Then $\textrm{gcd}(r_1,r_2)$ divides $p$, and it follows from the assumption $p$ being prime that $\textrm{gcd}(r_1,r_2)=1$ or $p$. 
\end{proof}

\begin{proof}[Proof of (ii)]
By the same argument as in the proof of Theorem~\ref{thm:a1}, we have 
\begin{align*}
\det(X)X^{-1}&=2(4t)^{2t-1}\left((4t+1)X^\top-\begin{pmatrix}r_1J & -r_2 J \\ r_2J & r_1 J  \end{pmatrix}\right).
\end{align*}
Therefore the entries of $\det(X)X^{-1}$ are $\pm 2(4t)^{2t-1}(4t+1\pm x)$ where $x\in\{r_1,r_2\}$. 
Then the greatest common divisor of $4t+1\pm x$ where $x\in\{r_1,r_2\}$ is 
\begin{align*}
\text{gcd}(4t+1\pm r_1,4t+1\pm r_2)&=2\text{gcd}(4t+1,r_1,r_2)=\begin{cases}
2 & \text{ if } \textrm{gcd}(r_1,r_2)=1,\\
2p & \text{ if } \textrm{gcd}(r_1,r_2)=p.\\
\end{cases}
\end{align*}
Therefore 
\begin{align*}
d_{4t+1}(X)=\begin{cases}
4(4t)^{2t-1} & \text{ if } \textrm{gcd}(r_1,r_2)=1,\\
4p(4t)^{2t-1} & \text{ if } \textrm{gcd}(r_1,r_2)=p,
\end{cases}
\end{align*}
and thus 
\begin{align*}
x_{4t+2}=\det(X)/d_{4t+1}(X)=\begin{cases}
2t(4t+1) & \text{ if } \textrm{gcd}(r_1,r_2)=1,\\
2tp & \text{ if } \textrm{gcd}(r_1,r_2)=p.\qedhere
\end{cases}
\end{align*}
\end{proof}

\begin{proof}[Proof of (iii)]
For the case $x_{4t+2}=2t(4t+1)$, the proof is same as that of Theorem~\ref{thm:a1}(ii). 

For the case $x_{4t+2}=2tp$,  $x_{4t+1}$ divides $x_{4t+2}$. Thus $x_{4t+1}$ divides $2tp$. 
By $x_{4t+2}=2tp$ and $\det(X)=x_1\cdots x_{4t+2}$, we have that 
\begin{align*}
x^\prime_2 \cdots x^\prime_{4t+1}=pt^{2t-1}.
\end{align*} 
Since $\textrm{gcd}(p,t)=1$, $p$ divides $x^\prime_{4t+1}$.  
\end{proof}

\begin{proof}[Proof of (iv)]
For the case $x_{4t+2}=2t(4t+1)$, the proof is same as that of Theorem~\ref{thm:a1}(iii). 

For the case $x_{4t+2}=2tp$,  the proof is also similar to that of Theorem~\ref{thm:a1}(iii), but we include it here.
By (ii) and $\det(X)=x_1\cdots x_{4t+2}$, we have that $x^\prime_2 \cdots x^\prime_{4t+1}=pt^{2t-1}$. 
Furthermore we put $x^{\prime\prime}_{4t+1}=x_{4t+1}/p$, which is integer by (iii). 
Then $x^\prime_2 \cdots x^\prime_{4t}x^{\prime\prime}_{4t+1}=t^{2t-1}$. 
Suppose, for a contradiction, that $x^\prime_{2t+2}>1$. 
Let $q$ be any prime divisor of $x^\prime_{2t+2}$. 
Then $x^\prime_2 \cdots x^\prime_{4t} x^{\prime\prime}_{4t+1}=t^{2t-1}$ is divisible by $q^{2t}$.  
However this is impossible; since $t$ is square-free $t^{2t-1}$ cannot have $q^{2t}$ as a divisor. 
Therefore $x^\prime_{2t+2}=1$.
Furthermore
\begin{align}\label{ineq:31}
x^\prime_{2t+3} \cdots x^\prime_{4t} x^{\prime\prime}_{4t+1}=t^{2t-1}.
\end{align} 
By (ii) we have $x^{\prime\prime}_{4t+1}\leqslant t$ and thus $x^\prime_i\leqslant t$ for $i\in\{2t+3,\ldots,4t\}$. 
Therefore, by \eqref{ineq:31}, we have $x^{\prime\prime}_{4t+1}= t$ and $x^\prime_{i}= t$ for all $i\in\{2t+3,\ldots,4t\}$. 
\end{proof}

Barba~\cite{Bar33} showed that for a $\{1,-1\}$-matrix $X$ of order $n \equiv 1\pmod{2}$, Hadamard's inequality can be strengthened to
\begin{align}\label{ineq:B}
|\det(X)| \leqslant (2n-1)^{1/2} (n-1)^{(n-1)/2}.
\end{align}
Moreover, there exists~\cite{CRCHAndbook2007} a $\{1,-1\}$-matrix achieving equality in \eqref{ineq:B} if and only if $n \equiv 1\pmod{4}$ and there exists a $\{1,-1\}$-matrix $B$ such that
\begin{align}\label{eq:b}
BB^\top=B^\top B=(n-1)I+J. 
\end{align}
A matrix $B$ is called a \textbf{Barba matrix} if it satisfies \eqref{eq:b}.  
It is known that if  $R$ is a Barba matrix of order $n$, then the matrix $\left (\begin{smallmatrix}
	R & R\\
	-R^\top & R^\top
	\end{smallmatrix} \right )$
is an EW matrix of order $2n$. 
We therefore pose the following question about Smith normal forms of EW matrices constructed from Barba matrices. 
\begin{problem}
	Does there exist a Barba matrix $R$ such that the Smith normal form of $\left (\begin{smallmatrix}
	R & R\\
	-R^\top & R^\top
	\end{smallmatrix} \right )$ is not equal to	
			$\operatorname{diag}[1,\underbrace{2,\dots,2}_{2t},\underbrace{2t,\dots,2t}_{2t-1},2t\sqrt{8t+1}]$.
\end{problem}

%


\section{The Smith normal form of an EW tournament matrix}
\label{sec:snftournament}

In this section we determine the Smith normal form of an EW tournament matrix.

\begin{theorem}\label{thm:snfa}
The Smith normal form of an EW tournament matrix of order $4t+1$ is 
\begin{align*}
		&\operatorname{diag}[\underbrace{1,\dots,1}_{2t+2},\underbrace{t,\dots,t}_{2t-2},t^2(4t-1)]. 
\end{align*}
\end{theorem}

Recall that for an integer matrix $M$, we denote by $d_i(M)$ the greatest common divisor of all $i\times i$ minors of $M$ and $d_0(M)=1$.
The first lemma we need is elementary.

\begin{lemma}[page 33 of \cite{MN}]\label{lem:snfprod}
Let $M$ and $N$ be invertible $n\times n$ integer matrices, and let $m_i$ be the invariant factors of $M$ and 
$p_i$ be the invariant factors of $MN$. 
Then $m_i$ divides $p_i$ for each $i$.
\end{lemma}

Now we state and prove two lemmas that we will use to prove Theorem~\ref{thm:snfa}.

\begin{lemma}\label{lem:snfa2}
	Let $A$ be an EW tournament matrix of order $4t+1$.
Suppose that a prime $p$ divides $t$. 
Then $\rank_p(A)=2t+2$. 
\end{lemma}
\begin{proof}
First we show that $\rank_p(A) \leqslant 2t+2$. 
Combine \eqref{eq:ew1} and Sylvester's rank inequality \cite[0.4.5 Rank inequalities (c)]{HJ} to obtain
\begin{align*}
    2\rank_p(A)&=\rank_p(A)+\rank_p(A^\top)\\
    &\leqslant \rank_p(AA^\top)+4t+1\\
    &=4t+4.
\end{align*}
Thus $\rank_p(A)\leqslant 2t+2$. 

Next we show that $\rank_p(A) \geqslant 2t+2$.
Let $d = \dim( \mathrm{rowsp}_p(A)\cap \mathrm{rowsp}_p(A^\top) )$ be the dimension of the intersection of row spaces.
Then
\begin{align*}
4t&=\rank_p(J-I)=\rank_p(A+A^\top)\\
&\leqslant \rank_p(A)+\rank_p(A^\top)-d\\
&=2\rank_p(A)-d, 
\end{align*}
that is, we have that $2t+d/2\leqslant \rank_p(A)$. 
Hence it suffices to show that $d\geqslant 3$.

The equations \eqref{eq:ew1} and \eqref{eq:ew2} together with the assumption that $p$ divides $t$ yields 
\begin{align*}
W \subseteq \mathrm{rowsp}_p(A) \text{ and } W \subseteq \mathrm{rowsp}_p(A^\top), 
\end{align*}
where 
\begin{align*}
W=\mathrm{span}\{({\bf 1}^\top_t,{\bf 1}^\top_t,{\bf 1}^\top_{a},{\bf 1}^\top_{2t+1-a}), ({\bf 1}^\top_t,-{\bf 1}^\top_t,{\bf 0}^\top_{a},{\bf 0}^\top_{2t+1-a}),({\bf 0}^\top_t,{\bf 0}^\top_t,{\bf 1}^\top_{a},-{\bf 1}^\top_{2t+1-a})\}. 
\end{align*}
Therefore 
\begin{align*}
W \subseteq \mathrm{rowsp}_p(A)\cap \mathrm{rowsp}_p(A^\top), 
\end{align*}
and thus $d \geqslant 3$.
\end{proof}

\begin{lemma}\label{lem:a2a}
	Let $A$ be an EW tournament matrix of order $4t+1$ and let $l_1,\dots,l_{4t+1}$ be the invariant factors of $A^2+A$.
Then $l_{4t} = t$ and $l_{4t+1} = t^2(16t^2-1)$. 
\end{lemma}
\begin{proof}
	Define $\alpha_\pm := t (8 t+3\pm 2 s)+(1\pm s)/2$,  $\beta_\pm := (2 t+1) (2 t+(1\pm s)/2)$, and $\gamma_\pm := t(4 t+1\pm s)$.
	It is straightforward to deduce that the matrix $t^{2}(16t^2-1)(A^2+A)^{-1}$ is equal to
	\begin{align*}
	-t(16t^2-1)I_{4t+1}+\left(
	\begin{array}{cccc}
	 t (8 t+1)J_t & tJ & \gamma_+ J & \gamma_- J \\
	 (5 t+1)J & t (8 t+1)J_t & \beta_+ J & \beta_- J \\
	 \beta_+ J & \gamma_+ J & \alpha_+ J_{a} & tJ \\
	 \beta_- J & \gamma_- J & tJ &  \alpha_-J_{2t+1-a} \\
	\end{array}
	\right),
	\end{align*}
	where $s=\sqrt{8t+1}$.
Hence, to prove the lemma, it suffices to show that 
\begin{align*}
d_1(t^2(16t^2-1)(A^2+A)^{-1})&=1, \\
d_2(t^2(16t^2-1)(A^2+A)^{-1})&=t(16t^2-1). 
\end{align*}
It is easy to see that $t$ and $5t+1$ are (coprime) entries of $t^2(16t^2-1)(A^2+A)^{-1}$, belonging to $(1,2)$-block and $(2,1)$-block. 
Thus $d_1(t^2(16t^2-1)(A^2+A)^{-1})=1$. 
 
All the possible $2\times 2$ minors of $t^2(16t^2-1)(A^2+A)^{-1}$ have the form $\pm t(16t^2-1) \alpha$ where $\alpha$ is an element of the set $S$ given by
\begin{align*}
S= &\left\{0,t,t+1,2 t,2 t+1,4 t+1,5 t+1,t (8 t+1),8 t^2-3 t-1,t \left(16 t^2-16 t-3\right),\right .\\
&16 t^3-16 t^2+t+1,2 t (4 t\pm s),t (4 t+1\pm s),t (4 t-1\pm s),\\
&2 t \left(8 t^2-8 t-1\pm s\right),16 t^3-16 t^2-7 t-1\pm s(4 t+1),\\
&\frac{2 t+1\pm s}{2},\frac{2 t-1\pm s}{2},\frac{4 t+1\pm s}{2},\frac{6 t+1\pm s}{2},\frac{12 t+1\pm s}{2},\\
&\frac{(2 t-1) (4 t+1\pm s)}{2},\frac{(2 t+1) (4 t-1\pm s)}{2},\frac{(2 t+1)(4 t+1\pm s)}{2},\\
&\left . \frac{16 t^2-4 t-1\pm s}{2},\frac{16 t^2+6 t+1\pm s(4 t+1)}{2},\frac{16 t^2-2 t-1\pm s(4 t+1)}{2}\right\}. 
\end{align*}

Note that the greatest common divisor of the numbers in the set $S$ is $1$. 
Therefore $d_{4t}(t^2(16t^2-1)(A^2+A)^{-1})=t(16t^2-1)$. 
\end{proof}

We are now ready to prove Theorem~\ref{thm:snfa}.

\begin{proof}[Proof of Theorem~\ref{thm:snfa}]
Let $A$ be an EW tournament matrix and let $a_1,\ldots,a_{4t+1}$ be the invariant factors of $A$.   

\paragraph{Claim 1:} $a_{2t+2}=1$. 
Assume that there exists a prime $p$ dividing $a_{2t+2}$. It follows from Lemma~\ref{lem:minor} and Lemma~\ref{lem:snfa1} that $a_1\cdots a_{4t+1}=\det(A)=t^{2t}(4t-1)$. 
Thus $p$ is either 
\begin{enumerate}[(1)]
\item a divisor of $t$, or 
\item not a divisor of $t$ and a divisor of $4t+1$. 
\end{enumerate} 
For the case (1), Lemma~\ref{lem:snfa2} contradicts the fact that 
\begin{align*}
\rank_p(A)=\max\{i\mid \text{$p$ does not divide $a_i$}\}\leqslant 2t+1.
\end{align*} 
For the case (2), observe that $p$ is odd and hence $p \geqslant 3$.
Furthermore $p$ is a divisor of $a_i$ for $i=2t+3,\ldots,4t+1$. 
Thus $p^{2t}$ divides $4t-1$.
But this is impossible since $p^{2t}>4t-1$ for all $t\geqslant 1$.  
Therefore we have $a_{2t+2}=1$. 
This proves Claim 1. 

\paragraph{Claim 2:} $a_{4t+1}$ divides $t^2(4t-1)$. 
By Lemma~\ref{lem:snfprod}, and Lemma~\ref{lem:a2a}, we have that $a_{4t+1}$ divides $l_{4t+1}=t^{2}(16t^2-1)$. 
On the other hand $a_{4t+1}$ divides $\det(A)=t^{2t}(4t-1)$. 
Therefore $a_{4t+1}$ divides $t^2(4t-1)$.  
This proves Claim~2. 

\paragraph{Claim 3:} $a_{4t}$ divides $t$. 
By Lemma~\ref{lem:snfprod}, and Lemma~\ref{lem:a2a}, we have that $a_{4t}$ divides $l_{4t}=t$. Therefore $a_{4t}$ divides $t$.  
This proves Claim~3. 

Now, by Lemma~\ref{lem:snfa1}, we have
\begin{align}\label{eq:cm}
a_{2t+3}\cdots a_{4t+1}=t^{2t}(4t-1). 
\end{align}
Since $a_i$ divides $a_{i+1}$ for each $i$ and $a_{4t}\leqslant t$, by Claim~3, 
we have that $a_{i}\leqslant t$ for $i=2t+3,\dots,4t-1$, from which with $a_{4t+1}\leqslant t^2(4t-1)$ it follows that 
\begin{align}\label{eq:cm1}
a_{2t+3}\cdots a_{4t+1}\leqslant t^{2t}(4t-1).
\end{align}
By \eqref{eq:cm}, equality must hold in \eqref{eq:cm1}. 
Therefore $a_{2t+3}=\dots=a_{4t}=t$ and $a_{4t+1}=t^2(4t-1)$.    
\end{proof}

\section*{Acknowledgements} 
\label{sec:acknowledgements}

We thank the referees for their careful reading of the manuscript and their helpful comments.



\end{document}